\documentclass[12pt]{article}
\usepackage{fullpage}
\usepackage{amsthm,amsmath,amsfonts,amssymb}
\usepackage{xcolor}
\usepackage{bbm}
\usepackage{enumerate}
\usepackage{hyperref}
\usepackage[capitalize]{cleveref}
\usepackage{tikz}
\usetikzlibrary{arrows,shadows}

\usepackage{nicematrix}

\newcommand{\footremember}[2]{%
   \footnote{#2}
    \newcounter{#1}
    \setcounter{#1}{\value{footnote}}%
}
\newcommand{\footrecall}[1]{%
    \footnotemark[\value{#1}]%
}
\newtheorem{theorem}{Theorem}
\newtheorem{definition}{Definition}
\newtheorem{lemma}{Lemma}
\newtheorem{corollary}{Corollary}
\newtheorem{claim}{Claim}

\newcommand{\ignore}[1]{}

\newif\ifnotesw\noteswtrue% T to show box & marginal notes; F suppresses.
   {\ifnotesw\marginpar[\hfill\(\top\)]{\(\top\)}\fi}%
   {\ifnotesw\marginpar[\hfill\(\bot\)]{\(\bot\)}\fi}

\newcommand{\mnote}[1]%
    {\ifnotesw\marginpar%
        [{\scriptsize\begin{minipage}[t]{\marginparwidth}
        \raggedleft#1%
                        \end{minipage}}]%
        {\scriptsize\begin{minipage}[t]{\marginparwidth}
        \raggedright#1%
                        \end{minipage}}%
    \fi}
\newcommand{\bra}[1]{\langle #1 |}

\newcommand{\ket}[1]{| #1 \rangle}

\newcommand{\braket}[2]{\langle #1 | #2 \rangle}
\newcommand{\lip}[2]{\langle #1 , #2 \rangle}
\newcommand{\ketbra}[2]{\ket{#1}\bra{#2}}
\newcommand{\etal}{{\em et al.}~}
\newcommand{\iverson}[1]{\lbrack\!\lbrack #1 \rbrack\!\rbrack}

\newcommand{\ZZ}{\mathbb{Z}}

\newcommand{\CC}{\mathbb{C}}
\newcommand{\OO}{\mathcal{O}}
\newcommand{\Exp}{\mathbb{E}}
\newcommand{\one}{\mathbbm{1}}
\DeclareMathOperator{\Sp}{Sp}
\DeclareMathOperator{\Tr}{Tr}
\DeclareMathOperator{\diag}{diag}

\DeclareMathOperator{\cir}{Circ}
\DeclareMathOperator{\Arg}{Arg}
\newcommand{\cay}{\Gamma}

\title{Uniform Mixing in Chiral Quantum Walks}
\author{Luke Levine\footremember{clarkson}{Computer Science, Clarkson University. Email: \{levinelj,mesapaj,musticbt,tino,tuckergs\}@clarkson.edu.}
\and
Jessy Jacob Mesapam\footrecall{clarkson}
\and
Benjamin Mustico\footrecall{clarkson}
\and
Christino Tamon\footrecall{clarkson}
\and
Gabriel Tucker\footrecall{clarkson}
\and
Hanmeng Zhan\footremember{wpi}{Computer Science, Worcester Polytechnic Institute. Email: hzhan@wpi.edu.}
}
\date{\today}
\begin{document}
\maketitle

\begin{abstract}
This paper studies uniform mixing in continuous-time quantum walks.
We show that for some unitary signing $\sigma$, the complete graph $K^\sigma_n$ has probabilistic uniform mixing. 
In contrast, Ahmadi \etal (2003) proved that {\em no} complete graph has uniform mixing except for $K_2$, $K_3$, and $K_4$. 
Our technique is based on a stopping rule for quantum walks which reduces global to local uniform mixing.
As a corollary, we found an orientation of $H(n,4)$ that mixes to uniform faster than any other Hamming graphs,
which improves a result of Godsil and Zhan (2019).
We also show that there are infinite families of oriented circulants with {\em average} uniform mixing.
This is a chiral violation of a {\em No-Go} theorem due to Godsil (2013) which states that no graph 
has average uniform mixing except for $K_2$.
\vspace{0.1in}
\end{abstract}

%%%%%%%%%%%%%%%%%%%%%%%%%%%%%%%%%%%%%%%%%%%%%%%%%%%%%%%%%%%%%%%%%%%%%%%%%%%%%%%%%%%%%%%%%%%%%%%%%%%%%
\section{Introduction}

Quantum walk on graphs is a natural quantum analogue of random walk on graphs.
It has proved useful in developing quantum communication protocols (for example, entanglement generation)
and quantum algorithms for important computational problems (for example, Grover search).

Given a graph $X$ with adjacency matrix $A(X)$, the continuous-time quantum walk on $X$ is given by
the time-varying unitary matrix $U_X(t) = \exp(-iA(X)t)$. The mixing matrix $M_X(t)$ is obtained
from $U_X(t)$ by taking entry-wise squared norm (which is the quantum mechanical Born rule for 
the quantum to classical transition). In uniform mixing at time $t$, we require all entries of $M_X(t)$ 
to be the constant $1/n$, where $n$ is the number of vertices of $X$.
This implies that if we measure the quantum walk at time $t$, then the corresponding probability
distribution is uniform over the vertices of $X$.

In their seminal result, Moore and Russell \cite{mr02} initiated the study of uniform mixing in 
continuous-time quantum walks. They observed that a quantum walk on the $n$-cube exactly mixes to
uniform at time $\frac{\pi}{4}n$. In contrast, a classical random walk on the $n$-cube requires $\Omega(n\log n)$
to converge to the uniform distribution in the limit. This shows a quantum advantage for uniform mixing
in both the quantitative (faster time) and qualitative (exact mixing versus mixing in the limit) aspects.

Another motivation for exploring uniform mixing is provided by the {\em W state} (see \cite{dvc00,c02}).
For three qubits, the W state is given by 
$e_{100}+e_{010}+e_{001}$.
%$\ket{100}+\ket{010}+\ket{001}$.
Here, $e_{100}$ denote the tensor product $e_1 \otimes e_0 \otimes e_0$ where $e_0$ and $e_1$ 
are the standard basis vectors of $\CC^2$.
Note that this can be generalized to $n$ qubits and is one of the target states for uniform mixing.
This useful entangled quantum state has interesting applications for secure communication in quantum networks 
(such as anonymous transmission \cite{lmw18} and leader election \cite{dp06}).
In certain applications, the W state is known to be more robust compared to the well-known GHZ state 
$e_{000}+e_{111}$.
%$\ket{000}+\ket{111}$.

It is known that classical random walks rapidly mix on the complete graphs. On the other hand,
Ahmadi \etal \cite{abtw03} observed curiously that no complete graph $K_n$ has uniform mixing except
for $K_2,K_3$, and $K_4$. In this work, we show that for any $K_n$ there is a unitary signing $\sigma$
so that $K_n^\sigma$ has uniform mixing in expected $n^{3/2}$ time. Our result is based on a simple
stopping rule which reduces uniform mixing to {\em local} uniform mixing. This is reminiscent of
stopping rules for exact mixing in Markov chains (see \cite{lw95}) and is inspired by a similar
technique used in Xie \etal \cite{xkt23} (on anti-Zeno effects in perfect state transfer). 

As an implicit corollary, we show there is an oriented Hamming graph $H(n,4)$ which has uniform mixing
at time $\pi/3\sqrt{3}$. This is {\em faster} than the mixing times of any unoriented Hamming graphs,
which include $H(n,2)$, $H(n,3)$, and $H(n,4)$, and the Cartesian product of $K_{1,3}$ with itself 
(which has uniform mixing at time $2\pi/3\sqrt{3}$).

\begin{figure}[h]
\begin{center}
\begin{tabular}{|c||c|c|} \hline
Graph family			& Mixing Time 		& Reference \\ \hline
$H(n,2)$				& $\pi/4$ 			& Moore-Russell \cite{mr02} \\
$H(n,3)$				& $2\pi/9$ 			& folklore \\
$H(n,4)$				& $\pi/4$ 			& folklore \\
$(K_{1,3})^{\Box n}$ 	& $2\pi/3\sqrt{3}$ 	& Godsil-Zhan \cite{gz17} \\
oriented $H(n,4)$		& $\pi/3\sqrt{3}$	& this work \\ \hline
\end{tabular}
\caption{Quantum uniform mixing times.}
\end{center}
\end{figure}

Our speedup in mixing time is a simple consequence of the signing on $K_4$ which improves the mixing 
profile of $K_{1,3}$. More specifically, the signing of $K_4$ allows us to treat any vertex as the
conical vertex in $K_{1,3}$ (where the mixing time is $\pi/3\sqrt{3}$). However, if we start 
at a pendant vertex in $K_{1,3}$, we incur a slower mixing time of $2\pi/3\sqrt{3}$.

For the second part of this work, we consider {\em average} mixing in quantum walks.
The average mixing matrix $\widehat{M}_X$ is the Cesaro limit of its mixing matrix,
that is, the limit of $\frac{1}{T}\int_{0}^{T} M_X(t) dt$ as $T$ tends to infinity (see \cite{aakv01}). 
As this limit always exists, $\widehat{M}_X$ provides a long-term profile of the quantum walk on $X$.
So, it is natural to ask which graphs $X$ admit uniform mixing for its $\widehat{M}_X$.

Unfortunately, Godsil \cite{g13} proved that $K_2$ is the {\em only} graph with average uniform mixing.
In contrast, we show that there are infinite families of oriented circulants with average uniform mixing.
Recently, Sin \cite{s25} discovered examples of oriented non-abelian Cayley graphs with uniform mixing.
Here, we show that {\em no} oriented non-abelian Cayley graphs have {\em average} uniform mixing.
These results suggest an interesting interplay between chirality and group-theoretic properties of
the underlying Cayley graphs.

%%%%%%%%%%%%%%%%%%%%%%%%%%%%%%%%%%%%%%%%%%%%%%%%%%%%%%%%%%%%%%%%%%%%%%%%%%%%%%%%%%%%%%%%%%%%%%%%%%%%%
\section{Preliminaries}

We briefly summarize terminology from graph theory and matrix theory that will be used throughout.
Our standard references are Godsil and Royle \cite{gr01} and Horn and Johnson \cite{hj13}.
We adopt the Iverson notation $\iverson{\Phi}$ for the characteristic function of 
a logical statement $\Phi$. 

For any integer $n \ge 1$, we let $I_n$ and $J_n$ denote the identity and all-one matrices of order $n$, 
respectively. The all-one vector is denoted $\one_n$. Whenever the context is clear, we will omit the 
subscript $n$ for simplicity. The unit vector $e_a$ satisfies $e_a(b) = \iverson{a=b}$.

Given a simple undirected graph $X$, we denote its vertex set as $V(X)$ and its edge set as $E(X)$. 
The order of $X$ refers to the number of vertices in $X$, namely $|V(X)|$.
The adjacency matrix $A(X)$ of $X$ is usually defined as $A(X)_{a,b} = \iverson{ab \in E(X)}$.

The complement of a graph $X$ is a graph $\overline{X}$ whose vertex set is $V(X)$ but with edge set
given by $E(\overline{X}) = \{ab \ : \ ab \not\in E(X)\}$. Note that $A(\overline{X})=J-I-A(X)$.
The {\em join} of two graphs $X$ and $Y$, denoted $X + Y$, is obtained from disjoint union of $X$ and $Y$
by adding the edges $(x,y)$ for all $x \in X$ and $y \in Y$. 
The Cartesian product $X \Box Y$ of $X$ and $Y$ is a graph whose vertex set is $V(X) \times V(Y)$ where
$(x_1,y_1)$ is adjacent to $(x_2,y_2)$ if either $x_1=x_2$ and $(y_1,y_2) \in E(Y)$, or, 
$(x_1,x_2) \in E(X)$ and $y_1=y_2$.
The adjacency matrix of $X \Box Y$ is given by $A(X) \otimes I + I \otimes A(Y)$.

We use $K_n$ to denote the complete graph of order $n$ while $\overline{K}_n$ denotes its complement.
The {\em cone} of a graph $X$, which is $K_1 + X$, will be denoted as $\widehat{X}$.
The {\em claw} $K_{1,n}$ is the cone of $\overline{K}_n$.

A {\em signing} of $X$ is a function $\sigma:E(X) \rightarrow \CC^\times$ defined on the edges of $X$.
We denote the resulting ``signed'' graph as $X^\sigma$ with adjacency matrix 
$A(X^\sigma) = \sigma(ab)\iverson{ab \in E(X)}$ which is required to be Hermitian.
We call the signing $\sigma$ unitary if $|\sigma(ab)|=1$.
If the range of $\sigma$ is $\pm i$, we call $X^\sigma$ an {\em oriented} graph.
An oriented graph whose underlying graph is $K_n$ can be viewed as a {\em tournament}.

Two signed graphs $X^{\sigma_1}$ and $X^{\sigma_2}$ obtained from the same underlying graph $X$
is called {\em switching equivalent}, denoted $X^{\sigma_1} \cong X^{\sigma_2}$, 
if there is a complex diagonal matrix $D$ so that
\begin{equation} \label{eqn:switching-equivalent}
	A(X^{\sigma_2}) = D^\dagger A(X^{\sigma_1}) D.
\end{equation}
See \cref{fig:signing} for an example of a signing on $K_4$ which is switching equivalent to $K_1 + \vec{K}_3$.
That is,
\[
	\begin{pmatrix}
	0 & -i & -i & -i \\
	i & 0  & -i & i  \\
	i & i  & 0  & -i \\
	i & -i & i  & 0
	\end{pmatrix}
	\cong	
	\begin{pmatrix}
	0 &  1 & 1  & 1 \\
	1 &  0 & -i & i  \\
	1 & i  & 0  & -i \\
	1 & -i & i  & 0
	\end{pmatrix}.
\]

\begin{figure}[b]
\begin{center}
\begin{minipage}[t]{0.45\textwidth}
\vspace{-0.1\baselineskip}
\[
	\begin{pNiceMatrix}[first-row,first-col]
	  & 0 & 1 & 2 & 3 \\
	0 & 0 & -i & -i & -i \\
	1 & i & 0  & -i & i  \\
	2 & i & i  & 0  & -i \\
	3 & i & -i & i  & 0
	\end{pNiceMatrix}.
\]
\end{minipage}%\hfill
\begin{minipage}[t]{0.45\textwidth}
\vspace{0pt}
\begin{tikzpicture}[-,>=stealth',auto,node distance=1.75cm,thick,
        main node/.style={circle,draw,font=\bfseries},
        main edge/.style={-,>=stealth'}, weird edge/.style={red,very thick}]
\tikzset{every loop/.style={min distance=5mm, in=65, out=115}}
% Oriented K4 
  \node[main node, circular drop shadow] (0) {$0$};
  \node[main node, circular drop shadow] (1) [below right of=0] {$3$};
  \node[main node, circular drop shadow] (2) [below left of=0] {$1$};
  \node[main node, circular drop shadow] (3) [below right of=2] {$2$};

  \path[->]
    (0) edge node {} (1)
    (0) edge node {} (1)
    (0) edge node {} (2)
    (0) edge node {} (3)
    (1) edge node {} (2)
    (2) edge node {} (3)
    (3) edge node {} (1);
\end{tikzpicture}
\end{minipage}
\caption{A unitary signing of $K_4$.
This is switching equivalent to $K_1 + \vec{K}_3$.}
\label{fig:signing}
\end{center}
\end{figure}
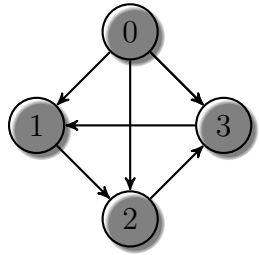

\subsection{Quantum walks}

Given a graph $X$, let $A(X)$ denote its adjacency matrix whose spectral decomposition is
\[
    A(X) = \sum_r \lambda_r E_r
\]
where $E_r$ are the orthogonal projectors onto the $\lambda_r$-eigenspace.
The continuous-time quantum walk on $X$ is defined as the unitary matrix 
\[
	U_X(t) = \exp(-itA(X)) = \sum_r e^{-i\lambda_r t} E_r
\]
where the second equality followed from the spectral theorem.

We observe that quantum walks on switching equivalent graphs are similar.
This holds as
\begin{equation} \label{eqn:switching-equivalent-qwalks}
\exp(-itA(X^{\sigma_2})) = D^\dagger \exp(-itA(X^{\sigma_1})) D
\end{equation}
by \cref{eqn:switching-equivalent} and $D^\dagger D = I$.

The {\em mixing} matrix of a quantum walk is given by
\[
    M_X(t) = U_X(t) \circ \overline{U_X(t)}
\]
where $A \circ B$ denotes the Schur product of matrices defined by $(A \circ B)_{jk} = A_{jk}B_{jk}$.
Note that the $(a,b)$-entry of $M_X(t)$ encodes the probability of measuring
$b$ at time $t$ when starting at $a$ (according to the Born rule of quantum mechanics).

\begin{definition} 
A graph $X$ of order $n$ has {\em uniform mixing} at time $t$ if $M_X(t) = \frac{1}{n}J_n$.
On the other hand, we say $X$ has {\em local} uniform mixing from vertex $a$ at time $t$ if
$M_X(t)e_a = \frac{1}{n}\one_n$.
\end{definition}

We state a useful property of uniform mixing under Cartesian graph products.

\begin{lemma} \label{lemma:closure} (Moore and Russell \cite{mr02}) \\
If $X$ has local uniform mixing from vertex $a$ at time $t$ and
$Y$ has local uniform mixing from vertex $b$ at time $t$, then
$X \Box Y$ has local uniform mixing from vertex $(a,b)$ at time $t$.
\end{lemma}

\paragraph{Measured walks}
In our quantum walks, we allow partial measurements at a vertex.
Formally, if $a \in V(X)$, the partial measurement relative to vertex $a$ is given by
the set $\frak{M} = \{P_0,P_1\}$ of projection operators, where
\[
	P_1 = e_a e_a^T
%\ketbra{a}{a}, 
	\ \ \
	P_0 = I - e_a e_a^T.
%I - \ketbra{a}{a}.
\]
If we apply the measurement operator $\frak{M}$ to a state 
${\psi} = c_a e_{a} + \sum_{b \neq a} c_b e_{b}$, 
%$\ket{\psi} = c_a\ket{a} + \sum_{b \neq a} c_b\ket{b}$, 
where $|c_a|^2 + \sum_{b \neq a} |c_b|^2 = 1$,
then the outcome is
\[
	\frak{M}{\psi} \mapsto
	\left\{
	\begin{array}{ll}
	1 & \mbox{ with probability $|c_a|^2$ } \\
	0 & \mbox{ with probability $1-|c_a|^2$}
	\end{array}\right.
\]
The corresponding {\em post-measurement} states for each outcome are given by
\[
	\left\{
	\begin{array}{ll}
	e_{a} & \mbox{ if outcome is $1$,} \\
	\frac{1}{\sqrt{1-|c_a|^2}}\sum_{b \neq a} c_b e_{b} & \mbox{ if outcome is $0$.}
	\end{array}\right.
\]
For more information on quantum measurements, see Nielsen and Chuang \cite{nc00} (page 87).

\subsection{Quotients}

We summarize useful facts about equitable partitions for {\em Hermitian} matrices (see Haemers \cite{h95}).
Let $A$ be a $n \times n$ Hermitian matrix and let $\pi = \bigsqcup_{j=1}^{m} V_j$ be a partition of $[n]$.
We call each $V_j$ a cell of $\pi$.
Let $A[V_j,V_k]$ be the submatrix of $A$ whose rows are indexed by elements in $V_j$ and whose
columns are indexed by elements in $V_k$. Let $n_j = |V_j|$ denote the size of cell $V_j$.

\begin{definition} (Complex equitable partition)
The partition $\pi$ is called {\em equitable} 
if each submatrix $A[V_j,V_k]$ has a {\em constant row sum} and a {\em constant column sum}.
That is, for each $j,k$, there are complex constants $r_{jk},c_{jk} \in \CC$ so that
\begin{align}
A[V_j,V_k]\one_{n_k} 	&= r_{jk}\one_{n_j} \ \ \ \mbox{ (constant row sum)} \\
\one_{n_j}^TA[V_j,V_k] 	&= c_{jk}\one_{n_k}^T \ \ \ \mbox{ (constant column sum).}
\end{align}
\end{definition}

As $A$ is Hermitian, $r_{jk} = \overline{c_{kj}}$. In particular, this implies that $r_{jj}$ and $c_{jj}$ are both real.
Let $S$ be the normalized characteristic $n \times m$ matrix of the partition $\pi$ where 
\begin{equation} \label{eqn:char-mat}
S_{jk} = \frac{1}{\sqrt{n_k}}\iverson{j \in V_k}.
\end{equation}
As $S$ is the matrix representation of $\pi$, we often refer to $S$ as {\em the} equitable partition.

\begin{lemma} \label{lemma:char-mat}
The characteristic matrix $S$ satisfies the following:
\begin{enumerate}[(i)]
\item $S^T S = I_m$.
\item $SS^T$ is a block diagonal $m \times m$ matrix $\diag(D_1,\ldots,D_m)$, where $D_j = \frac{1}{n_j}J_{n_j}$.
\item $[SS^T,A] = 0$.
\end{enumerate}

\begin{proof}
{\em (i)} This follows since the $j$th column of $S$ is the normalized characteristic vector of $V_j$ 
and the subsets $V_j$ are disjoint. 
{\em (ii)} This holds since the outer product of the normalized characteristic vector of $V_j$ 
is the normalized all-one matrix $B_j = \frac{1}{n_j}J_{n_j}$.
{\em (iii)} It suffices to show $B_j$ commutes with $A[V_j,V_j]$ for each $j$.
Note $B_j A[V_j,V_j] = c_{jj} J_{n_j}$ and $A[V_j,V_j]B_j = r_{jj} J_{n_j}$.
As $r_{jj} = c_{jj}$, we are done.
\end{proof}
\end{lemma}

\par\noindent
The {\em quotient} matrix of $A$ relative to $\pi$ is the $m \times m$ matrix defined by
\[
	B := S^T A S.
\]

\begin{claim}
The quotient matrix $B$ is a Hermitian $m \times m$ matrix whose $jk$-entry is given by
\[
	B_{jk} = \sqrt{|r_{jk}| |c_{jk}|} \ e^{i\Theta}
\]
where $\Theta = \Arg(r_{jk}) = \Arg(c_{jk})$.

\begin{proof}
The $jk$-entry of $B$ is given by
\[
	B_{jk} = e_j^T(S^T A S)e_k = (Se_j)^T A (Se_k).
\]
Therefore, as $Se_j = \frac{1}{\sqrt{n_j}}\one_{n_j}$ and similarly for $Se_k$, we have
\[
	B_{jk} = c_{jk} \sqrt{\frac{n_k}{n_j}} = r_{jk} \sqrt{\frac{n_j}{n_k}}.
\]
This shows that $r_{jk}$ and $c_{jk}$ have the same principal argument, that is, $\Arg(r_{jk}) = \Arg(c_{jk})$,
since $n_j$ and $n_k$ are both real.
Also, since $n_j r_{jk} = n_k c_{jk}$, we have
\[
	B_{jk} = \sqrt{|r_{jk}| |c_{jk}|} \ e^{i\Arg(r_{jk})}.
\]

To show $B$ is Hermitian, note that
\[
	B_{kj} = \sqrt{|r_{kj}| |c_{kj}|} \ e^{i\Arg(r_{kj})}
	= \sqrt{|c_{jk}| |r_{jk}|} \ e^{-i\Arg(r_{jk})}
	= \overline{B_{jk}}
\]
as $r_{kj} = \overline{r_{jk}}$. 
\end{proof}
\end{claim}

\begin{lemma} \label{lemma:quotient-walk}
The quantum walk on $B$ and the quantum walk on $A$ are related as follows:
\[
	e^{-iBt} = S^T e^{-iAt} S.
\]
Moreover, if $V_i = \{u\}$ and $V_j = \{v\}$ are singleton cells, then
\[
	\lip{e_{V_j}}{e^{-iBt} e_{V_i}} = \lip{e_{v}}{e^{-iAt} e_{u}}.
\]

\begin{proof}
The first claim holds since $(S^T A S)^k = S^T A^k S$, for every $k \ge 0$. 
The second claim follows from $Se_{V_i} = e_u$ and $Se_{V_j} = e_v$.
\end{proof}
\end{lemma}

%%%%%%%%%%%%%%%%%%%%%%%%%%%%%%%%%%%%%%%%%%%%%%%%%%%%%%%%%%%%%%%%%%%%%%%%%%%%%%%%%%%%%%%%%%%%%%%%%%%%%
\section{Conical reduction}

The following impossibility result for uniform mixing is known for cliques.

\begin{theorem} (Ahmadi \etal \cite{abtw03}) \label{thm:nogo-clique}
$K_n$ has uniform mixing if and only if $n=2,3,4$. 
\end{theorem}

In contrast to \cref{thm:nogo-clique}, we show that the complete graphs admit probabilistic\footnote{In the theory
of randomized algorithms, a {\em Las Vegas} algorithm is always correct but whose running time is a random variable, while
a {\em Monte Carlo} algorithm has a bounded running time but might make mistakes. See Motwani and Raghavan's excellent book \cite{mr95}.}
uniform mixing (of Las Vegas type).

\begin{theorem} \label{thm:conical}
Let $X$ be a graph of order $n$ whose Hermitian adjacency matrix $A$ has zero as a simple eigenvalue 
with the corresponding eigenvector $\one_n$. 
Then, $\widehat{X}$ has uniform mixing in expected $n^{3/2}$ time.

\begin{proof}
The adjacency matrix of $\widehat{X}$ is given by
\[
	A(\widehat{X}) = 
	\begin{pmatrix}
	0 & \one_n^T \\
	\one_n & A(X)
	\end{pmatrix}
\]
whose eigenvalues are $\pm\sqrt{n}$ and $\theta \in \Sp(A(X))$ with corresponding eigenvectors
\[
	\psi_{\pm} = \frac{1}{\sqrt{2}}
	\begin{pmatrix}
	1 \\
	\pm\frac{1}{\sqrt{n}}\one_n
	\end{pmatrix},
	\ \ \
	\psi_{\theta} =
	\begin{pmatrix}
	0 \\
	z_\theta
	\end{pmatrix}
\]
where $z_\theta$ is a normalized eigenvector of $A(X)$ corresponding to a nonzero eigenvalue $\theta$.

In the quantum walk, the start vertex is either conical or not.
If the start vertex is the conical vertex $0$, then
\[
	e^{-i A(\widehat{X})t} e_{0} 
	= \sum_\pm \frac{e^{\mp i\sqrt{n}t}}{\sqrt{2}} \psi_{\pm} 
	= \begin{pmatrix} \cos(\sqrt{n}t) \\ -\frac{i}{\sqrt{n}}\sin(\sqrt{n}t) \one_n \end{pmatrix}
\]
which shows uniform mixing at $t_0 = \frac{1}{\sqrt{n}}\arccos(\frac{1}{\sqrt{n+1}}) \approx \frac{\pi}{2\sqrt{n}}$.

On the other hand, if the start vertex is not conical, say $e_u$, then:
\[
	e^{-i A(\widehat{X})t} e_{u} 
	= \sum_\pm \frac{\pm e^{\mp i\sqrt{n}t}}{\sqrt{2n}} \psi_{\pm} 
		+ \sum_{\theta \neq 0} \lip{e_u}{\psi_\theta} e^{-i\theta t}\psi_{\theta}.
\]
This implies that the amplitude at the conical vertex is given by
\[
	\lip{e_0}{e^{-i A(\widehat{X})t} e_{u}} 
	= \sum_\pm \frac{\pm e^{\mp i\sqrt{n}t}}{2\sqrt{n}} 
	= -\frac{i}{\sqrt{n}}\sin(\sqrt{n}t).
\]
By linearity, this extends to any normalized vector that is orthogonal to $e_0$.
At time $t_1 = \pi/2\sqrt{n}$, the quantum walk is at the conical vertex with probability $1/n$.
Thus, if we perform a partial measurement $\{P_0=e_{0}e_{0}^T, P_1=I-e_{0}e_{0}^T\}$, the expected waiting time
for the measurement to return $P_0$ is $n$. Once the quantum walk arrives at the conical vertex, we appeal to the first case
and will achieve uniform mixing after time $t_0$. 
Let $T$ be the random variable for the total time until uniform mixing. Then,
\[
	\Exp[T] = n \times \frac{\pi}{2\sqrt{n}} + t_0 = \OO(\sqrt{n}).
\]
But, after scaling the adjacency matrix by $n$ (so that it has bounded norm), we obtain a uniform mixing time of $\OO(n^{3/2})$.
\end{proof}
\end{theorem}

The proof technique of \cref{thm:conical} is based on Xie \etal \cite{xkt23} 
where a partial measurement was used to reduce perfect state transfer to fractional revival.

%%%%%%%%%%%%%%%%%%%%%%%%%%%%%%%%%%%%%%%%%%%%%%%%%%%%%%%%%%%%%%%%%%%%%%%%%%%%%%%%%%%%%%%%%%%%%%%%%%%%%

\begin{figure}[h]
\begin{center}
\begin{tikzpicture}[-,>=stealth',auto,node distance=1.75cm,thick,
        main node/.style={circle,draw,font=\bfseries},
        main edge/.style={-,>=stealth'}, weird edge/.style={red,very thick}]

\tikzstyle{every node}=[draw, thick, shape=circle, style={circular drop shadow}];
\path (0.0,2.0) node [scale=0.9] (a0) {};
\path (+2.0,+1.0) node [scale=0.9] (a1) {};
\path (+2.0,-1.0) node [scale=0.9] (a2) {};
\path (0.0,-2.0) node [scale=0.9] (a3) {};
\path (-2.0,-1.0) node [scale=0.9] (a4) {};
\path (-2.0,+1.0) node [scale=0.9] (a5) {};

\tikzstyle{every edge}=[draw, ->, green];
\draw[->,thick]
	(a0) -- (a1);
\draw[->,thick]
	(a0) -- (a2);
\draw[->,thick]
	(a0) -- (a3);
\draw[->,thick]
	(a0) -- (a4);
\draw[->,thick]
	(a0) -- (a5);

\draw[->,gray]
    (a1) -- (a2);
\draw[->,gray]
    (a2) -- (a3);
\draw[->,gray]
    (a3) -- (a4);
\draw[->,gray]
    (a4) -- (a5);
\draw[->,gray]
    (a5) -- (a1);

\draw[->,gray]
    (a1) -- (a3);
\draw[->,gray]
    (a3) -- (a5);
\draw[->,gray]
    (a5) -- (a2);
\draw[->,gray]
    (a2) -- (a4);
\draw[->,gray]
    (a4) -- (a1);
\end{tikzpicture}
\quad\quad\quad\quad
\begin{tikzpicture}[-,>=stealth',auto,node distance=1.75cm,thick,
        main node/.style={circle,draw,font=\bfseries},
        main edge/.style={-,>=stealth'}, weird edge/.style={red,very thick}]

\tikzstyle{every node}=[draw, thick, shape=circle, style={circular drop shadow}];
\path (0.0,2.0) node [scale=0.9] (a0) {};
\path (+2.0,+1.0) node [scale=0.9] (a1) {};
\path (+2.0,-1.0) node [scale=0.9] (a2) {};
\path (0.0,-2.0) node [scale=0.9] (a3) {};
\path (-2.0,-1.0) node [scale=0.9] (a4) {};
\path (-2.0,+1.0) node [scale=0.9] (a5) {};

\draw[-]
	(a0) -- (a1)
	(a0) -- (a2)
	(a0) -- (a3)
	(a0) -- (a4)
	(a0) -- (a5);
\end{tikzpicture}
\caption{The chiral ghost trick: $K_6$ morphs into $K_{1,5}$.}
\label{fig:cone}
\end{center}
\end{figure}
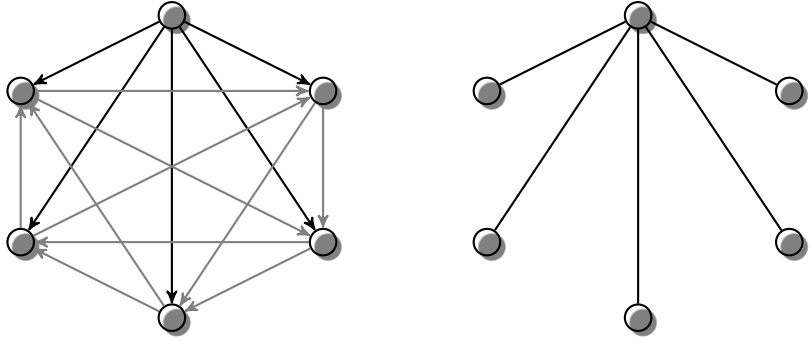
%%%%%%%%%%%%%%%%%%%%%%%%%%%%%%%%%%%%%%%%%%%%%%%%%%%%%%%%%%%%%%%%%%%%%%%%%%%%%%%%%%%%%%%%%%%%%%%%%%%%%

In contrast to \cref{thm:nogo-clique}, we have the following corollary of \cref{thm:conical}.

\begin{corollary} \label{cor:clique}
For any $n \ge 5$, there is a unitary signing $\sigma$ so that the complete graph $K_n^\sigma$ has uniform mixing 
(in expected $\OO(n^{3/2})$ time).

\begin{proof}
Note $K_{n+1} = K_1 + K_n$. We consider two cases based on the parity of $n \ge 4$.

\smallskip
\par\noindent{\em Case 1}: $n$ is odd.
We use a circulant signing $A = \cir(0,-i,i,-i,i,\ldots,-i,i)$.
Clearly, $A\one_n = 0$ and zero is a simple eigenvalue. So, \cref{thm:conical} applies.

\smallskip
\par\noindent{\em Case 2}: $n$ is even
We apply a unitary signing using all roots of unity. Let $\omega = e^{2\pi i/(2m+1)}$.
For $k=1,\ldots,m$, let
\[
	B_0 = \begin{pmatrix} 0 & 1 \\ 1 & 0 \end{pmatrix},
	\ \ \
	B_k = 
	\begin{pmatrix} 
	\omega^{2k-1} & \omega^{2k} \\
	\omega^{2k}   & \omega^{2k-1}
	\end{pmatrix},
	\ \ \
	\tilde{B}_k = 
	\begin{pmatrix} 
	\omega^{2k} & \omega^{2k-1} \\
	\omega^{2k-1}   & \omega^{2k}
	\end{pmatrix},
\]
Then, we consider the skew block circulant given by
\[
	C = 
	\begin{pmatrix}
	B_0             & B_1         & B_2         & \ldots & B_m     \\
	\tilde{B}_m     & B_0         & B_1         & \ldots & B_{m-1} \\
	\tilde{B}_{m-1} & \tilde{B}_m & B_0         & \ldots & B_{m-2} \\
	\vdots          & \vdots      & \vdots      & \ldots & \vdots  \\
	\tilde{B}_1     & \tilde{B}_2 & \tilde{B}_3 & \ldots & B_{m-1}
	\end{pmatrix}
\]
From its definition, the first two rows of $C$ are orthogonal to the all-one vector as it contains the distinct
entries $0,1,\omega,\ldots,\omega^{2m}$. The second row was obtained from the first row by pairwise disjoint swaps
within each $2\times 2$ submatrix $B_k$.
Each pair of subsequent rows is obtained by a permutation of the $B_k$, for $k=0,1,\ldots,m$, and thus are also
orthogonal to the all-one vector.
Finally, observe that $C$ is Hermitian since $\tilde{B}_k = B^\dagger_{m-k+1}$, for $k=1,\ldots,m$.
Hence, \cref{thm:conical} also applies to this case.
\end{proof}
\end{corollary}

See \cref{fig:cone} for an example of the signing used in \cref{cor:clique}.

\medskip
\par\noindent{\em Example}:
For $n=7$ and $n=8$, the unitary signings used in the proof of \cref{cor:clique} are given by
\[
	\begin{pmatrix}
	0  & -i &  i & -i &  i & -i &  i \\
	i  &  0 & -i &  i & -i &  i & -i \\
	-i &  i &  0 & -i &  i & -i &  i \\
	i  & -i &  i &  0 & -i &  i & -i \\
	-i &  i & -i &  i &  0 & -i &  i \\
	i  & -i &  i & -i &  i &  0 & -i \\
	-i &  i & -i &  i & -i &  i &  0 
	\end{pmatrix},
	\ \ \
	\begin{pmatrix}
	0        & 1        & \omega   & \omega^2 & \omega^3 & \omega^4 & \omega^5 & \omega^6 \\
	1        & 0        & \omega^2 & \omega   & \omega^4 & \omega^3 & \omega^6 & \omega^5 \\
	\omega^6 & \omega^5 & 0        & 1        & \omega   & \omega^2 & \omega^3 & \omega^4 \\ 
	\omega^5 & \omega^6 & 1        & 0        & \omega^2 & \omega   & \omega^4 & \omega^3 \\
	\omega^4 & \omega^3 & \omega^6 & \omega^5 & 0        & 1        & \omega   & \omega^2 \\
	\omega^3 & \omega^4 & \omega^5 & \omega^6 & 1        & 0        & \omega^2 & \omega   \\
	\omega^2 & \omega   & \omega^4 & \omega^3 & \omega^6 & \omega^5 & 0        & 1        \\
	\omega   & \omega^2 & \omega^3 & \omega^4 & \omega^5 & \omega^6 & 1        & 0
	\end{pmatrix}.
\]
The first matrix is a $\pm i$-circulant matrix whereas the second matrix is a block skew-circulant
matrix whose entries are seventh roots of unity.
\bigskip

The following result generalizes Zhan's observation that $K_{1,3}$ has uniform mixing at time $2\pi/3\sqrt{3}$
(see \cite{gz17}). 

\begin{corollary} \label{cor:claw}
For any $n \ge 1$, $K_{1,n}$ has probabilistic uniform mixing in expected $n^{3/2}$ time.
\end{corollary}

\begin{corollary} \label{cor:euler}
For any nonsingular Eulerian graph $X$, $\widehat{X}$ has probabilistic uniform mixing.

\begin{proof}
An Eulerian graph has an Eulerian orientation which can be used to define a $\pm i$-signing of $X$.
Now, we apply \cref{thm:conical}.
\end{proof}
\end{corollary}

\par\noindent{\em Remark}:
If we allow self-loops and integer edge weights in graphs, then {\em any} connected graph has uniform mixing.
The Laplacian of a graph $X$ is defined as $L(X) = D(X) - A(X)$, where $D(X)$ is the diagonal
degree matrix of $X$ (where $D(X)_{a,a}$ is the degree of vertex $a$ in $X$). 
Note $L(X)$ is a weighted graph defined over $X$ which contains self-loops.

\begin{corollary} \label{cor:laplacian}
For any connected graph $X$, the integer-weighted graph $\widehat{L(X)}$ has probabilistic uniform mixing.
\end{corollary}

%%%%%%%%%%%%%%%%%%%%%%%%%%%%%%%%%%%%%%%%%%%%%%%%%%%%%%%%%%%%%%%%%%%%%%%%%%%%%%%%%%%%%%%%%%%%%%%%%%%%%
\section{Chiral speedup}

Given an oriented graph $X$, note that the cone $\widehat{X} = K_1 + X$ is also an oriented graph
as the edges incident to the conical vertex can oriented with either $+i$ (outgoing) or $-i$ (incoming).

\begin{claim} \label{claim:k4}
There is a signing $\sigma$ of $K_4$ so that $K_4^\sigma$ has uniform mixing at $\pi/3\sqrt{3}$.

\begin{proof}
We sign $K_4$ using $\sigma$ as follows:
\[	
	A(K_4^\sigma) = 
	\begin{pmatrix}
	0 & -i & -i & -i \\
	i & 0  & -i & i  \\
	i & i  & 0  & -i \\
	i & -i & i  & 0
	\end{pmatrix}.
\]
Observe that $K_4^\sigma$ is switching equivalent to $K_1 + \vec{K}_3$ (by a $-i$ switching at the conical vertex).

But, $K_1 + \vec{K}_3$ and $K_{1,3}$ have the same quotient under an equitable partition which places the conical 
vertex in a singleton cell. Thus, by \cref{lemma:quotient-walk}, their quantum walks are similar. 
Now, recall that $K_{1,3}$ has uniform mixing at $\pi/3\sqrt{3}$ starting at the conical vertex (see \cite{gz17}).

Finally, observe that $K_1 + \vec{K}_3$ has a switching automorphism that sends the conical vertex to any other vertex.
This shows we have a {\em global} uniform mixing.
\end{proof}
\end{claim}

This is a speedup of Zhan's observation that $K_1 + \overline{K}_3$ has uniform mixing at time $2\pi/3\sqrt{3}$.
Recall that the Hamming graph $H(n,d)$ is defined as the $n$-fold Cartesian product of the complete graph $K_d$,
namely, $H(n,d) = (K_d)^{\Box n}$. It is called the Rook's graph if $n=2$ as it encodes all legal movements of a 
rook on a $n \times n$ chessboard.

\begin{theorem} \label{thm:hamming}
For $n \ge 1$, there is a signed $H(n,4)$ with uniform mixing at time $\pi/3\sqrt{3}$.

\begin{proof}
Let $\sigma$ be a signing of $K_4$ so that $K_4^\sigma$ is the oriented cone $K_1 + \vec{K}_3$.
By \cref{claim:k4}, $K_4^\sigma$ has uniform mixing at $\pi/3\sqrt{3}$.
As $H(n,4) = (K_4^\sigma)^{\Box n}$ and uniform mixing is closed under Cartesian product by \cref{lemma:closure},
the claim follows.
\end{proof}
\end{theorem}

%%%%%%%%%%%%%%%%%%%%%%%%%%%%%%%%%%%%%%%%%%%%%%%%%%%%%%%%%%%%%%%%%%%%%%%%%%%%%%%%%%%%%%%%%%%%%%%%%%%%%
\begin{figure}[h]
\begin{center}
\begin{tikzpicture}[-,>=stealth',auto,node distance=1.75cm,thick,
        main node/.style={circle,draw,font=\bfseries},
        main edge/.style={-,>=stealth'}, weird edge/.style={red,very thick}]
\tikzset{every loop/.style={min distance=5mm, in=65, out=115}}
% K4 
  \node[main node, circular drop shadow] (0) {};
  \node[main node, circular drop shadow] (1) [below right of=0] {};
  \node[main node, circular drop shadow] (2) [below left of=0] {};
  \node[main node, circular drop shadow] (3) [below right of=2] {};

  \path[main edge]
    (0) edge node {} (1)
    (0) edge node {} (1)
    (0) edge node {} (2)
    (0) edge node {} (3)
    (1) edge node {} (2)
    (2) edge node {} (3)
    (3) edge node {} (1);
\end{tikzpicture}
\quad \quad \quad
\begin{tikzpicture}[-,>=stealth',auto,node distance=1.75cm,thick,
        main node/.style={circle,draw,font=\bfseries},
        main edge/.style={-,>=stealth'}, weird edge/.style={red,very thick}]
\tikzset{every loop/.style={min distance=5mm, in=65, out=115}}
% K(1,3)
  \node[main node, circular drop shadow] (0) {};
  \node[main node, circular drop shadow] (1) [below right of=0] {};
  \node[main node, circular drop shadow] (2) [below left of=0] {};
  \node[main node, circular drop shadow] (3) [below right of=2] {};

  \path[main edge]
    (0) edge node {} (1)
    (0) edge node {} (2)
    (0) edge node {} (3);
\end{tikzpicture}
\quad \quad \quad
\begin{tikzpicture}[-,>=stealth',auto,node distance=1.75cm,thick,
        main node/.style={circle,draw,font=\bfseries},
        main edge/.style={-,>=stealth'}, weird edge/.style={red,very thick}]
\tikzset{every loop/.style={min distance=5mm, in=65, out=115}}
% Oriented K4 
  \node[main node, circular drop shadow] (0) {};
  \node[main node, circular drop shadow] (1) [below right of=0] {};
  \node[main node, circular drop shadow] (2) [below left of=0] {};
  \node[main node, circular drop shadow] (3) [below right of=2] {};

  \path[->]
    (0) edge node {} (1)
    (0) edge node {} (1)
    (0) edge node {} (2)
    (0) edge node {} (3)
    (1) edge node {} (2)
    (2) edge node {} (3)
    (3) edge node {} (1);
\end{tikzpicture}
\caption{A tale of three uniform mixing times:
(a) $K_4$ at $\pi/4$.
(b) $K_{1,3}$ at $2\pi/3\sqrt{3}$. 
(c) Oriented $K_4$ at $\pi/3\sqrt{3}$.
}
\label{fig:k4-oriented}
\end{center}
\end{figure}
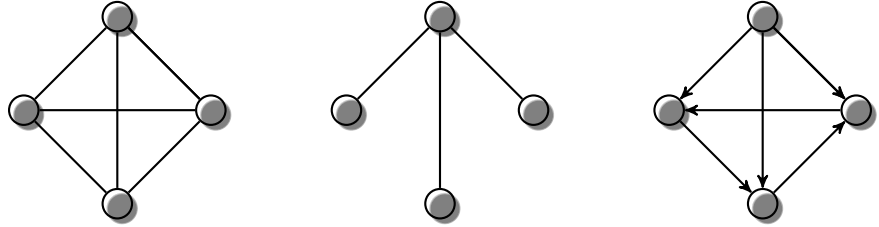
%%%%%%%%%%%%%%%%%%%%%%%%%%%%%%%%%%%%%%%%%%%%%%%%%%%%%%%%%%%%%%%%%%%%%%%%%%%%%%%%%%%%%%%%%%%%%%%%%%%%%

Among the Hamming graphs with uniform mixing, the oriented $H(n,4)$ in \cref{thm:hamming} 
has the fastest uniform mixing time. In comparison, the uniform mixing times of 
$H(n,2)$ and $H(n,4)$ are $\pi/4$ while $H(n,3)$ is $2\pi/9$.

See \cref{fig:k4-oriented} for graphs on four vertices with uniform mixing.

%%%%%%%%%%%%%%%%%%%%%%%%%%%%%%%%%%%%%%%%%%%%%%%%%%%%%%%%%%%%%%%%%%%%%%%%%%%%%%%%%%%%%%%%%%%%%%%%%%%%%
\section{Average mixing}

The {\em average mixing} matrix $\widehat{M}_X$ of a graph $X$ is the Cesaro limit of its mixing matrix $M_X(t)$.
That is, $\widehat{M}_X$ is given by
\[
    \widehat{M}_X = \lim_{T \rightarrow \infty} \frac{1}{T} \int_{0}^{T} M_X(t) dt
\]
Using the spectral theorem, we have that (see \cite{g13,cggz18})
\begin{equation} \label{eqn:schur-squared}
    \widehat{M}_X = \sum_r E_r \circ \overline{E}_r.
\end{equation}
As $\widehat{M}_X$ is a doubly stochastic matrix, it is a Markov chain constructed from taking 
the ergodic limit of a continuous-time {\em quantum} process $U_X(t)$.

\begin{definition}
A graph $X$ of order $n$ has {\em average uniform mixing} if $\widehat{M}_X = \frac{1}{n}J_n$.
\end{definition}

Godsil proved the following strong {\em No-Go} theorem for average uniform mixing on graphs.

\begin{theorem} \label{thm:k2} (Godsil \cite{g13})
A graph $X$ has average uniform mixing if and only if $X=K_2$.
\end{theorem}

In stark contrast to \cref{thm:k2}, we show there are infinite families of oriented graphs with average
uniform mixing.

\begin{theorem} \label{thm:distinct}
Any $(\ZZ/n\ZZ)$-circulant with distinct eigenvalues has average uniform mixing.

\begin{proof}
Let $X$ be a circulant of order $n$ with adjacency matrix $A(X) = \sum_{r=0}^{n-1} \lambda_r E_r$, where
the eigenvalues $\lambda_r$ are all distinct, and
$E_r = \ketbra{\zeta_r}{\zeta_r}$ where $\ket{\zeta_r}$ is the $r$-th column of the Fourier matrix
whose $j$-th entry is $\braket{j}{\zeta_r} = \frac{1}{\sqrt{n}}\zeta_n^{jr}$, for $j,r = 0,1,\ldots,n-1$.
Then,
\[
    \widehat{M}_X = \sum_{r=0}^{n-1} E_r \circ \overline{E}_r = \sum_{r=0}^{n-1} \frac{1}{n^2} J_n = \frac{1}{n}J_n
\]
as $\overline{E}_r = E_{-r}$.
\end{proof}
\end{theorem}

Recall we call a graph {\em oriented} if its edges have $\{\pm i\}$ weights.

\begin{corollary}
The oriented odd-cycle has average uniform mixing.

\begin{proof}
The eigenvalues of an oriented cycle $\vec{C}_n$ of order $n$ are $\lambda_k = 2\sin(2\pi k/n)$, for $k=0,1,\ldots,n-1$,
which are all distinct when $n$ is odd. As $\vec{C}_n$ is a $(\ZZ/n\ZZ)$-circulant, we may apply \cref{thm:distinct}.
\end{proof}
\end{corollary}

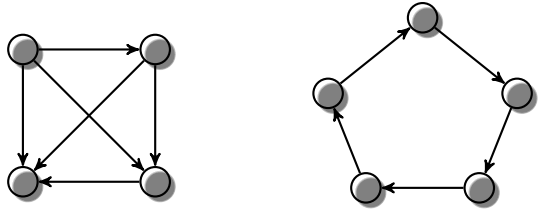
\begin{figure}[h]
\begin{center}
\begin{tikzpicture}[-,>=stealth',auto,node distance=1.75cm,thick,
        main node/.style={circle,draw,font=\bfseries},
        main edge/.style={-,>=stealth'}, weird edge/.style={red,very thick}]
\tikzset{every loop/.style={min distance=5mm, in=65, out=115}}
% Transitive K4 
  \node[main node, circular drop shadow] (0) {};
  \node[main node, circular drop shadow] (1) [right of=0] {};
  \node[main node, circular drop shadow] (2) [below of=1] {};
  \node[main node, circular drop shadow] (3) [left of=2] {};

  \path[->]
    (0) edge node {} (1)
    (0) edge node {} (2)
    (0) edge node {} (3)
    (1) edge node {} (2)
    (1) edge node {} (3)
    (2) edge node {} (3);
\end{tikzpicture}
\quad\quad\quad\quad
\begin{tikzpicture}[-,>=stealth',auto,node distance=1.75cm,thick,
        main node/.style={circle,draw,font=\bfseries},
        main edge/.style={-,>=stealth'}, weird edge/.style={red,very thick}]

\tikzstyle{every node}=[draw, thick, shape=circle, style={circular drop shadow}];
\path (0.0,2.25) node [scale=1.0] (a0) {};
\path (+1.25,+1.25) node [scale=1.0] (a1) {};
\path (+0.75,0.0) node [scale=1.0] (a2) {};
\path (-0.75,0.0) node [scale=1.0] (a3) {};
\path (-1.25,+1.25) node [scale=1.0] (a4) {};

\draw[->]
	(a0) -- (a1);
\draw[->]
	(a1) -- (a2);
\draw[->]
	(a2) -- (a3);
\draw[->]
	(a3) -- (a4);
\draw[->]
	(a4) -- (a0);
\end{tikzpicture}
\caption{Oriented graphs with average uniform mixing:
(a) Transitive tournament $T_{n}$.
(b) Oriented odd-cycle $\vec{C}_{2n+1}$.
}
\label{fig:average-mixing}
\end{center}
\end{figure}

Next, we observe a family of oriented graphs with average uniform mixing for any order.

\begin{corollary}
For any $n \ge 2$, the oriented skew $(\ZZ/n\ZZ)$-circulant (corresponding to the transitive tournament)
has average uniform mixing.

\begin{proof}
A skew-circulant matrix $A$ whose first row is given by $(a_0,a_1,\ldots,a_{n-1})$ can be written as
$A = \sum_{j=0}^{n-1} a_jP^j$, where $P$ maps $e_j$ to $e_{j-1}$, $j=1,\ldots,n$, and $e_0$ to $-e_{n-1}$.
Let $z_k$ be a normalized vector whose entries are given by
$e_j^T z_k = e^{ijk\pi/n}/\sqrt{n}$. The {\em flat} matrix $Z$ whose $k$-th column is $z_k$ is a common eigenbasis
for all skew-circulants (see Davis \cite{d94}).
For the acyclic tournament $A=(0,1,1,\ldots,1)$, the oriented $B=-iA$ has eigenvalues
$\lambda_j = \cot((2j+1)\pi/2n)$, $j=0,1,\ldots,n-1$, which are all {\em distinct}.
Thus, \cref{thm:distinct} applies.
\end{proof}
\end{corollary}

See \cref{fig:average-mixing} for small examples of circulants with average uniform mixing.

A graph $X$ has {\em universal} perfect state transfer if it has perfect state transfer between every pair of vertices;
that is, for all $a,b \in V(X)$, there is a time $t$ so that $|\bra{a}e^{-iA(X)t}\ket{b}| = 1$.

\begin{corollary}
Any $(\ZZ/n\ZZ)$-circulant with universal perfect state transfer has average uniform mixing.

\begin{proof}
Any graph with universal perfect state transfer must have distinct eigenvalues (see \cite{cfghst14}).
Now, apply \cref{thm:distinct}.
\end{proof}
\end{corollary}

\subsection{Non-abelian No-Go theorem}

Next, we consider Cayley graphs of non-abelian groups.
We will appeal to the continuous-time version of a result of Godsil and Zhan (see \cite{gz19}, Theorem 10.4).

\ignore{
\begin{theorem} \label{thm:trace-bound} (Godsil and Zhan \cite{gz19}, Theorem 10.4)\\
Suppose $X$ is a graph of order $n$ with spectral decomposition $A(X) = \sum_{k=1}^{d} \lambda_k E_k$
where the multiplicity of $\lambda_k$ is $m_k$. Then,
\[
    \Tr\widehat{M}_X \ge \frac{1}{n}\sum_{k=1}^{d} m_k^2 \ge 1.
\]
Moreover, the minimum is achieved whenever each eigenvalue is simple and $X$ is walk-regular.
\end{theorem}

Note that a graph $X$ is walk-regular if all of its vertices are cospectral.
}

\begin{lemma}
Given any graph $X$ of order $n$ with Hermitian adjacency matrix $A = \sum_r \lambda_r E_r$, we have
\[
    \Tr\widehat{M}_X \ \ge \ \frac{1}{n}\sum_{k=1}^{d} m_k^2 \ \ge \ 1
\]
where $m_r$ is the multiplicity of $\lambda_r$. Moreover, equality is achieved if each $E_r$ has constant diagonal
and each eigenvalue is simple.

\begin{proof}
By \cref{eqn:schur-squared}, we have
\[
	\Tr\widehat{M}_X 
	\ = \ \sum_r \Tr(E_r \circ \overline{E_r}) 
	\ = \ \sum_r \sum_a |(E_r)_{a,a}|^2.
\]
Apply Cauchy-Schwarz on the inner sum to get
\[
	\sum_a |(E_r)_{a,a}|^2 \ \ge \ \frac{1}{n}\left(\sum_a (E_r)_{a,a}\right)^2.
\]
But, $\Tr(E_r) = \sum_a (E_r)_{a,a} = m_r$, the multiplicity of eigenvalue $\lambda_r$. Therefore,
\[
	\Tr\widehat{M}_X \ \ge \ \frac{1}{n} m_r^2 \ \ge \ 1,
\]
where the last inequality follows by a standard optimization argument.
\end{proof}
\end{lemma}

If $\widehat{M}_X = \frac{1}{n}J_n$, then its trace achieves the lower bound of $1$.

For an irreducible representation $\rho$ of a finite group $G$, we let $d_\rho$ be
the degree (or dimension) of $\rho$.

\begin{theorem} \label{thm:babai} (Babai \cite{b79})\\
Let $X=\cay(G,S)$ be a Cayley graph over a finite group $G$ with a connection set $S$ where
$S = S^{-1}$ ($S$ is closed under taking inverses) and $1 \not\in S$ ($X$ has no self-loops).
Then,
\begin{enumerate}
\item $A(X) = \sum_{g \in S} R(g)$ where $R(g)$ is the left-regular representation of $G$.
\item There is a unitary matrix $U$ so that $U^\dagger A(X)U$ is a diagonal matrix with entries
    $\sum_{g \in S} R_\rho(g)$ (repeated $d_\rho$ times) for each irreducible representation $\rho$ of $G$.
\item The eigenvalues of $A(X)$ are given by the eigenvalues of $\sum_{g \in S} R_\rho(g)$ (each repeated
    $d_\rho$ times).
\end{enumerate}
\end{theorem}

\begin{corollary} \label{cor:nonabelian}
Let $X=\cay(G,S)$ be the Cayley graph of a nonabelian finite group $G$ with an inverse-closed
connection set $S$ which does not contain the identity. Then, $X$ has no average uniform mixing.

\begin{proof}
Any nonabelian finite group has an irreducible representation of degree larger than one.
By \cref{thm:babai}, $\cay(G,S)$ must have a repeated eigenvalue.
\end{proof}
\end{corollary}

%%%%%%%%%%%%%%%%%%%%%%%%%%%%%%%%%%%%%%%%%%%%%%%%%%%%%%%%%%%%%%%%%%%%%%%%%%%%%%%%%%%%%%%%%%%%%%%%%%%%%
\section{Conclusions} \label{section:open-questions}

In this work, we introduced a reduction from global to {\em local} uniform mixing using
weak measurements. This is similar to the reduction from perfect state transfer to
fractional revival (see \cite{xkt23}). It can also be viewed as a {\em stopping rule} for uniform mixing 
in quantum walks. We then combined this technique with {\em unitary signing} of graphs (see \cite{az20}). 
Using this sign-and-reduce technique, we show that there are oriented complete graphs 
with probabilistic (Las Vegas) uniform mixing. This complements the fact that no complete graph 
has uniform mixing (except $K_2,K_3,K_4$).

On the other hand, by unitary signing alone, we found families of oriented circulants with 
average uniform mixing which is a No-Go violation of Godsil's theorem \cite{g13}.
Here, chirality helps construct oriented graphs with {\em distinct} eigenvalues
in $(\ZZ/n\ZZ)$-circulants. But, we identified a barrier towards extending this
to nonabelian groups (due to a basic fact from group representation).

We conclude with a few questions for future explorations:
\begin{enumerate}[(a)]
\item Are there high-diameter trees with probabilistic uniform mixing?
\item Is there a speed limit for uniform mixing? In particular, can the methods developed by 
	Luby \etal \cite{lsz93} be used to optimized our Las Vegas algorithm?
\item Are there graphs with fast Monte Carlo uniform mixing?
	Here, we require algorithms with bounded running time but which achieves $\epsilon$-uniform mixing
	for some negligible $\epsilon > 0$.
\end{enumerate}

%%%%%%%%%%%%%%%%%%%%%%%%%%%%%%%%%%%%%%%%%%%%%%%%%%%%%%%%%%%%%%%%%%%%%%%%%%%%%%%%%%%%%%%%%%%%%%%%%%%%%
\section*{Acknowledgments}

We thank Ada Chan, Chris Godsil and Xiaohong Zhang for helpful discussions.
This work is supported by NSF grants CCF-2348399 and OSI-2427020.

%%%%%%%%%%%%%%%%%%%%%%%%%%%%%%%%%%%%%%%%%%%%%%%%%%%%%%%%%%%%%%%%%%%%%%%%%%%%%%%%%%%%%%%%%%%%%%%%%%%%%

%%%%%%%%%%%%%%%%%%%%%%%%%%%%%%%%%%%%%%%%%%%%%%%%%%%%%%%%%%%%%%%%%%%%%%%%%%%%%%%%%%%%%%%%%%%%%%%%%%%%%

\begin{thebibliography}{1}

\bibitem{aakv01}
D. Aharonov, A. Ambainis, J. Kempe, U. Vazirani,
Quantum Walks on Graphs,
{\em Proceedings of the Annual ACM Symposium on the Theory of Computing} (STOC), p50-59, 2001.

\bibitem{abtw03}
A. Ahmadi, R. Belk, C. Tamon, C. Wendler,
On Mixing in Continuous-time Quantum Walks on Some Circulant Graphs,
{\em Quantum Information and Computation} {\bf 3}(6):611-618, 2003.

\bibitem{az20}
N. Alon, K. Zheng,
Unitary signing and induced subgraphs of Cayley graphs of $\ZZ_2^n$,
{\em Advances in Combinatorics}, November 09, 2020.

\bibitem{b79}
L. Babai,
Spectra of Cayley Graphs,
{\em Journal of Combinatorial Theory B} {\bf 27}:180-189, 1979.

\bibitem{c02}
A. Cabello,
Bell's theorem with and without inequalities for the three-qubit Greenberger-Horne-Zeilinger and W states,
{\em Physical Review A} {\bf 65}:032108, 2002.

\bibitem{cfghst14}
S. Cameron, S. Fehrenbach, L. Granger, O. Hennigh, S. Shrestha, C. Tamon,
Universal State Transfer on Graphs,
{\em Linear Algebra and Its Applications} {\bf 455}:115-142, 2014.

%\bibitem{cgkst17}
%E. Connelly, N. Grammel, M. Kraut, L. Serazo, C. Tamon,
%Universality in Perfect State Transfer,
%{\em Linear Algebra and Its Applications} {\bf 531}:516-532, 2017.

\bibitem{cggz18}
G. Coutinho, C. Godsil, K. Guo, H. Zhan,
A New Perspective on the Average Mixing Matrix,
{\em Electronic Journal of Combinatorics} {\bf 25}(4), \#P4.14, 2018.

\bibitem{d94}
P. Davis, 
{\em Circulant Matrices}, second edition,
Chelsea, 1994.

\bibitem{dp06}
E. D'Hondt, P. Panangaden,
The computational power of the W and GHZ states,
{\em Quantum Information and Computation} {\bf 6}(2):173-183, 2006.

\bibitem{dvc00}
W. D\"{u}r, G. Vidal, J.I. Cirac,
Three qubits can be entangled in two inequivalent ways,
{\em Physical Review A} {\bf 62}:062314, 2000.

\bibitem{g13}
C. Godsil,
Average Mixing of Continuous Quantum Walks,
{\em Journal of Combinatorial Theory A} {\bf 120}:1649-1662, 2013.

\bibitem{gr01}
C. Godsil, G. Royle,
{\em Algebraic Graph Theory},
Springer, 2001.

\bibitem{gz19}
C. Godsil, H. Zhan,
Discrete-time Quantum Walks and Discrete Structures,
{\em J. Combinatorial Theory A} {\bf 167}:181-212, 2019.

\bibitem{gz17}
C. Godsil, H. Zhan,
Uniform Mixing on Cayley Graphs,
{\em Electronic Journal of Combinatorics} {\bf 24}, Issue 3, 2017.

\bibitem{h95}
W. Haemers,
Interlacing Eigenvalues and Graphs,
{\em Linear Algebra and Its Applications} {\bf 227-228}:593-616, 1995.

\bibitem{hj13}
R. Horn, C. Johnson,
{\em Matrix Analysis}, second edition,
Cambridge University Press, 2013.

\bibitem{lmw18}
V. Lipinska, G. Murta, S. Wehner,
Anonymous transmission in a noisy quantum network using the W state,
{\em Physical Review A} {\bf 98}:052320, 2018.

\bibitem{lw95}
L. Lovasz, P. Winkler,
Exact Mixing in an Unknown Markov Chain,
{\em Electronic Journal of Combinatorics} {\bf 2}, 1995.

\bibitem{lsz93}
M. Luby, A. Sinclair, D. Zuckerman,
Optimal speedup of Las Vegas algorithms,
{\em Information Processing Letters} {\bf 47}:173-180, 1993.

\bibitem{mr02}
C. Moore, A. Russell,
Quantum walks on the Hypercube,
{\em Proc. 6th International Workshop on Randomization and Approximation Techniques in Computer Science (RANDOM 2002)}, 
164-178, 2002.

\bibitem{mr95}
R. Motwani, P. Raghavan,
{\em Randomized Algorithms},
Cambridge University Press, 1995.

\bibitem{nc00}
M. Nielsen, I. Chuang,
{\em Quantum Computation and Quantum Information},
Cambridge University Press, 2000.

\bibitem{s25}
P. Sin,
Uniform mixing in continuous-time quantum walks on oriented, nonabelian Cayley graphs,
\href{https://arxiv.org/abs/2510.08376}{https://arxiv.org/abs/2510.08376}.

\bibitem{xkt23}
W. Xie, A. Kay, C. Tamon,
Breaking the Speed Limit of Perfect Quantum State Transfer,
{\em Physical Review A} {\bf 108}:012408, 2023.

\end{thebibliography}
\end{document}